\documentclass[11pt]{article}
\usepackage{HongLiu}

\title{A sharp $p$-biased product bound for $r$-cross-intersecting families}
\author{Fan Chang\thanks{School of Statistics and Data Science, Nankai University, Tianjin, China; and Extremal Combinatorics and Probability Group, Institute for Basic Science, Daejeon, South Korea. Email: \texttt{1120230060@mail.nankai.edu.cn}. Supported by the National Natural Science Foundation of China (NSFC) under grant 124B2019 and by the Institute for Basic Science (IBS-R029-C4).}, \quad 
Hong Liu\thanks{Extremal Combinatorics and Probability Group (ECOPRO), Institute for Basic Science (IBS), Daejeon, South Korea. Email: \texttt{hongliu@ibs.re.kr}. Supported by Institute for Basic Science IBS-R029-C4.},\quad
Miao Liu\thanks{Research Center for Mathematics and Interdisciplinary Sciences, Shandong University, Qingdao, China, and Extremal Combinatorics and Probability Group (ECOPRO), Institute for Basic Science (IBS), Daejeon, South Korea. Email: \texttt{liumiao10300403@163.com}. Supported by China Scholarship Council and Institute for Basic Science, IBS-R029-C4.}
}
\date{}
\begin{document}
\maketitle

\begin{abstract}
We prove a sharp product theorem for $r$-cross-intersecting families in the $p$-biased measure.  If $r\ge2$, $0\le p\le \frac{r-1}{r}$, and $\F_1,\dots,\F_r\subseteq 2^{[n]}$ are $r$-cross-intersecting, then
\[
    \prod_{i=1}^r \mu_p(\F_i)\le p^r .
\]
The bound is attained by a common $1$-star, and the range of $p$ is best possible.  In particular, this proves the equal-bias case of a conjecture of Frankl and Tokushige and, for $r=3$, confirms a conjecture of Tokushige.

We also prove a stability theorem: for \(r\ge3\), every near-extremal \(r\)-tuple is close, in
\(p\)-biased measure, to a common \(1\)-star, with an optimal linear dependence on the product deficit. 
The extremal proof uses a coordinatewise coupling at the critical bias together with an isoperimetric inequality for increasing families.  
The stability proof uses dual families and random ordered partitions to obtain Fourier concentration, then applies biased Friedgut--Kalai--Naor theorem to force the star structure.
\end{abstract}

\section{Introduction}

Intersection theorems are among the central results of extremal set theory.  The classical
Erd\H{o}s--Ko--Rado theorem~\cite{EKR1961} asserts that, for \(n\ge 2k\), every intersecting
family \(\F\subseteq\binom{[n]}{k}\) has size at most
    $\binom{n-1}{k-1},$
with equality attained by a \(1\)-star.  For arbitrary families of subsets of \([n]\), the
corresponding weighted problem is naturally formulated using the \(p\)-biased measure.  For
\(0\le p\le1\), let \(\mu_p\) denote the product measure on \(2^{[n]}\) under which each
element of \([n]\) is included independently with probability \(p\); equivalently,
\[
    \mu_p(\F)
    =
    \sum_{F\in\F}p^{|F|}(1-p)^{n-|F|}
    \qquad
    \text{for }\F\subseteq2^{[n]}.
\]
This probabilistic viewpoint goes back at least to the work of Fishburn, Frankl, Freed,
Lagarias and Odlyzko~\cite{FFFLO1986}; see also the survey and monograph of Frankl and
Tokushige~\cite{FT2016,FT2018book}.

For a single family, the \(p\)-biased Erd\H{o}s--Ko--Rado theorem and its refinements show
that stars are extremal in the natural range of parameters; see, for instance, the biased
complete intersection theorem of Filmus~\cite{Filmus2017}.  It is therefore natural to ask
whether the same phenomenon persists for several families under cross-intersection
assumptions. We say that families
$    \F_1,\ldots,\F_r\subseteq2^{[n]}$
are \emph{\(r\)-cross-intersecting} if
$    A_1\cap\cdots\cap A_r\ne\varnothing$
for every choice \(A_i\in\F_i\), \(i\in[r]\).\footnote{If some \(\F_i\) is empty, then the
condition is vacuously satisfied and all product bounds below are immediate.}
The most basic example is obtained by taking all families to be the same \(1\)-star
$\mathcal S_j:=\{A\subseteq[n]:j\in A\}.$
Then
$\prod_{i=1}^r\mu_{p_i}(\mathcal S_j)
    =
    \prod_{i=1}^r p_i.$
Frankl and Tokushige conjectured that this example is extremal.

\begin{conjecture}[Frankl--Tokushige~\cite{FT2018book}]
Let
 $\max_{i\in[r]}p_i\le \frac{r-1}{r},$
and let \(\F_1,\ldots,\F_r\subseteq2^{[n]}\) be \(r\)-cross-intersecting.  Then
\[
    \prod_{i=1}^r \mu_{p_i}(\F_i)\le \prod_{i=1}^r p_i.
\]
\end{conjecture}

Our first result proves the equal-bias case of this conjecture for all \(r\). In particular, for $r=3$, it confirms a conjecture of Tokushige~\cite[Conjecture 4.11]{Tokushige2022}.

\begin{theorem}\label{thm:main}
Let \(r\ge2\) and \(0\le p\le \frac{r-1}{r}\), and let
\(\F_1,\ldots,\F_r\subseteq2^{[n]}\) be \(r\)-cross-intersecting.  Then
\[
    \prod_{i=1}^r\mu_p(\F_i)\le p^r.
\]
\end{theorem}

The range in Theorem~\ref{thm:main} is best possible.  Indeed, suppose $p>(r-1)/r.$
For large \(n\), take
   all families to be $\left\{A\subseteq[n]: |A|>\frac{r-1}{r}n\right\}.$
By pigeonhole principle, these families are \(r\)-cross-intersecting. On the other hand, as
\(p>(r-1)/r\), by the law of large numbers,     $\mu_p(\F_i)\to1$
as $n\to\infty.$
Thus, for sufficiently large \(n\),
    $\prod_i\mu_p(\F_i)>p^r.$

The proof of Theorem~\ref{thm:main} rests on two ingredients.  The first is a
critical-bias additive bound: at the bias
$    p_*=\frac{r-1}{r},$
the sum of the measures of \(r\) cross-intersecting families is at most \(r-1\).
Related additive estimates were known in several settings.  For two families at
\(p=1/2\), the statement follows by complementation and is closely related to
Kleitman's non-disjointness theorem~\cite{Kleitman1966}.  For three families,
Tokushige~\cite[Theorem~7]{Tokushige2024Nontrivial3wise} proved a sharp result in
the range \(1/3\le p\le1/2\).  More generally, Gupta, Mogge, Piga and Sch\"ulke~\cite{GuptaMoggePigaSchuelke2023}
obtained sharp sum theorems for non-empty \(r\)-cross \(t\)-intersecting families,
which specialize to product measure only for \(p\le1/2\).
In the uniform setting, Frankl and Tokushige proved an analogous critical estimate for
\(r\)-cross-union \(n-k\)-uniform families~\cite{FT2011}.

These earlier mechanisms do not imply the critical-bias additive bound needed here.  Many
of them rely on shifting, compression, or complement-based arguments~\cite{Borg2016,Frankl1987,FT1992,FranklTokushige1998}, which naturally compare \(p\) with \(1-p\) and therefore
encounter a half-bias barrier.  Instead, we use a direct coordinatewise coupling: for each
coordinate \(a\in[n]\), independently choose \(R_a\in[r]\) uniformly at random, and include
\(a\) in all random sets \(X_i\) except \(X_{R_a}\).  Then each \(X_i\) has distribution
\(\mu_{p_*}\), while
    $X_1\cap\cdots\cap X_r=\varnothing$
deterministically.  Hence, for \(r\)-cross-intersecting families, the events \(X_i\in\F_i\)
cannot all occur, giving
\[
    \sum_{i=1}^r\mu_{p_*}(\F_i)\le r-1.
\]
This coupling is the point at which the proof reaches the sharp bias \(p_*=(r-1)/r\) directly
in the biased cube, rather than by the usual ``going to infinity and back'' route of
Dinur--Safra~\cite{DS2005} and Frankl--Tokushige~\cite{FT2003}.

The second ingredient is a comparison lemma for increasing families: if $0<p\le q<1$, then
    $\mu_p(\F)\le \mu_q(\F)^{\log p/\log q}$.
Equivalently, \(\log\mu_p(\F)/\log p\) is monotone in \(p\).  This monotonicity is
well known; see, for instance,~\cite[Theorem~2.38]{Grimmett} and
\cite[Lemma~2.6]{EllisKellerLifshitzJEMS2019}.  Analytically, it follows by differentiating
\(\log\mu_p(\F)/\log p\) and combining the Margulis--Russo formula~\cite{Margulis1974,Russo1981}
with the \(p\)-biased edge-isoperimetric inequality~\cite{EllisKellerLifshitz2019}.  In this
sense the comparison lemma is isoperimetric in nature: weak Harper's edge-isoperimetric
inequality can be recovered from it by differentiation, whereas differentiating the classical
small-set expansion bound for the \(\rho\)-noise operator does not give the sharp Harper
constant; see~\cite[Chapter~10]{ODonnell2014}.  In the proof below we give a short
self-contained induction proof.  Applying the comparison lemma with \(q=p_*\), and then
using AM--GM together with the additive bound above, gives the desired product bound.

We also prove a stability version of Theorem~\ref{thm:main}.  Closeness is measured in
\(p\)-biased symmetric difference; as usual,
    $\F\Delta\G=(\F\setminus\G)\cup(\G\setminus\F).$

\begin{theorem}\label{thm:r-cross-product-stability}
Let $r\ge3$.  There exist constants $\varepsilon_{\ref{thm:r-cross-product-stability}}(r)>0$ and $C_{\ref{thm:r-cross-product-stability}}(r)>0$ such that the following holds for all $0<p\le\frac{r-1}{r}$ and $0<\varepsilon\le\varepsilon_{\ref{thm:r-cross-product-stability}}(r)$.
Let $\F_1,\ldots,\F_r\subseteq2^{[n]}$ be $r$-cross-intersecting. If
$\prod_{i=1}^r \mu_p(\F_i) \ge (1-\varepsilon)p^r,$
then there exists a $1$-star $\mathcal S_j$ such that for all $i\in [r]$,
\[
\mu_p(\F_i \Delta \mathcal S_j)\le C_{\ref{thm:r-cross-product-stability}}(r)p\varepsilon.
\]
\end{theorem}

The assumption \(r\ge3\) is necessary at the critical point.  When \(r=2\) and \(p=1/2\), define the majority 
\[
\mathcal M_n:=
\begin{cases}
\{A\subseteq[n]: |A|>n/2\}, & \text{if } n \text{ is odd},\\[2mm]
\{A\subseteq[n]: |A|>n/2\}\cup
\{A\subseteq[n]: |A|=n/2\text{ and }1\in A\}, & \text{if } n \text{ is even}.
\end{cases}
\]
Then $\mathcal M_n$ is intersecting and $\mu_{1/2}(\mathcal M_n)=1/2$, so $\F_1=\F_2=\mathcal M_n$ is a non-star extremal pair.

The linear dependence on \(\varepsilon\) in Theorem~\ref{thm:r-cross-product-stability} is
best possible.  Let
$    \mathcal T_n:=\mathcal S_1\setminus\{\{1\}\}$
and take
$    \F_1=\mathcal T_n, \F_2=\cdots=\F_r=\mathcal S_1.$
If $    \varepsilon_n:=(1-p)^{n-1},$ then
    $\mu_p(\mathcal T_n)=p(1-\varepsilon_n)$
 and 
    $\mu_p(\mathcal T_n\Delta\mathcal S_1)=p\varepsilon_n.$
The product is \((1-\varepsilon_n)p^r\), while the distance from the extremal star is of order
\(p\varepsilon_n\).

The key ingredient in the proof of Theorem~\ref{thm:r-cross-product-stability} is the following additive stability result.

\begin{theorem}\label{lem:q-stability-r}
Let $r\ge 3$ be fixed, and let $p_*=\frac{r-1}{r}$. Then there exist constants $\eta_{\ref{lem:q-stability-r}}(r)>0$ and $C_{\ref{lem:q-stability-r}}(r)>0$ such that the following holds for all $0<\eta \le \eta_{\ref{lem:q-stability-r}}(r)$. Let $\A_1,\dots,\A_r\subseteq 2^{[n]}$ be $r$-cross-intersecting. If
$\sum_{i=1}^r \mu_{p_*}(\A_i)\ge r-1-\eta,$
then one of the following alternatives holds.
\begin{itemize}
    \item[(i)] There exists a $1$-star $\mathcal{S}_j\subseteq 2^{[n]}$ such that $\mu_{p_*}(\A_i\Delta \mathcal{S}_j)\le C_{\ref{lem:q-stability-r}}(r)\eta$ for every $i\in[r]$.

    \item[(ii)] There exists a unique index $i\in [r]$ such that $\mu_{p_*}(\A_i)\le C_{\ref{lem:q-stability-r}}(r)\eta$ and, for every $j\in[r]\setminus\{i\}$, $\mu_{p_*}(\A_j)\ge 1-C_{\ref{lem:q-stability-r}}(r)\eta$. 
\end{itemize}
\end{theorem}

We briefly explain the proof of Theorem~\ref{lem:q-stability-r}.  A direct inductive approach is possible at a qualitative level. One can slice the families with respect to a coordinate and apply induction. This gives a version of the same dichotomy, but forces a dimension-dependent restriction on the approximation parameter $\eta\le r^{-n}$. Such a loss would not yield a dimension-free quantitative stability theorem. Instead, we take a Fourier-analytic approach. First pass to the up-closures
    $\A_i^\uparrow$
and then to the dual families
\[
    \B_i:=\{B\subseteq[n]:[n]\setminus B\notin\A_i^\uparrow\}.
\]
The dualization changes the critical bias \(p_*=(r-1)/r\) into \(1/r\):
$    \mu_{1/r}(\B_i)=1-\mu_{p_*}(\A_i^\uparrow).$
Moreover, the \(r\)-cross-intersection condition becomes a covering condition for ordered
partitions: for every ordered partition
$    P_1,\ldots,P_r$
of \([n]\), at least one event \(P_i\in\B_i\) holds.

Now choose a random coloring \(x\in[r]^n\) and let
    $C_i(x):=\{a\in[n]:x_a=i\}.$
Then \(C_1(x),\ldots,C_r(x)\) form a random ordered partition of $[n]$, and each \(C_i(x)\) has
distribution \(\mu_{1/r}\).  Near equality in the critical-sum lemma implies that
$    \sum_{i=1}^r \mathbbm{1}_{\B_i}(C_i(x))$
is close in \(L^2\) to the constant function \(1\).  Expanding these functions in the
\(1/r\)-biased Fourier basis, a standard Gram-matrix calculation shows that the Fourier mass of
each \(\mathbbm{1}_{\B_i}\) above level \(1\) is \(O_r(\eta)\).  The biased
Friedgut--Kalai--Naor theorem~\cite{kindler2002property,kindler2002noise} then implies that
each \(\B_i\) is close to a constant family or to a \(1\)-star.  Finally, the ordered-partition
covering condition aligns these local alternatives, giving exactly the two possibilities in
Theorem~\ref{lem:q-stability-r}.

This use of FKN differs from previous Fourier-analytic stability arguments in
intersection theory.  For two cross-intersecting families, one can often encode the condition
using a graph operator and apply a Hoffman-type bound.  For intersection conditions involving
three or more simultaneous choices, high-dimensional Hoffman bounds have been developed and
used effectively, for example by Filmus, Golubev and Lifshitz~\cite{FGL2021} and by
Tokushige~\cite{Tokushige2022}.  Here we avoid this route: the random ordered partition and
the dual covering condition give the required Fourier concentration directly.

Starting from a near-extremal product in Theorem~\ref{thm:r-cross-product-stability}, we pass
to up-closures and use the comparison lemma to transfer the product lower bound to the
critical bias.  The stability form of AM--GM then forces all numbers
    $\mu_{p_*}(\F_i^\uparrow)$
to be close to \(p_*\).  This rules out alternative (ii) in Theorem~\ref{lem:q-stability-r}, and alternative (i) gives a common star at the critical bias.  A final application of the
comparison lemma on the two slices of the approximating star transfers the conclusion
back from \(p_*\) to the original bias \(p\).

\medskip\noindent\emph{Further consequences.}
In the concluding section we record several consequences and directions.  First, the biased
product theorem gives a short proof of the asymptotic \(k\)-uniform product bound, and the
biased stability theorem gives a rough \(k\)-uniform stability statement via the standard
up-closure and local LYM argument.  Second, the equal-bias theorem implies a partial
unequal-bias result for logarithmically balanced choices of \(p_1,\ldots,p_r\).  We also discuss
an index-hypergraph formulation of the product argument.

\medskip\noindent\emph{Organization.}
Section~\ref{sec:pre} contains notation and preliminary tools.  Section~\ref{sec:extremal}
proves the critical-bias additive bound and the comparison lemma, and then uses them to
prove Theorem~\ref{thm:main}.  Section~\ref{sec:stability} proves
Theorem~\ref{lem:q-stability-r} and Theorem~\ref{thm:r-cross-product-stability}.
Section~\ref{sec:concluding} records consequences, extensions, and open questions.

\section{Preliminaries}\label{sec:pre}

\subsection{Definitions and notation}
Let $[n]:=\{1,\dots,n\}$ and let $2^{[n]}$ denote the power set of $[n]$. A family $\F\subseteq 2^{[n]}$ is said to be \emph{increasing} if given $A\in \F$ and $A\subseteq B$ we also have $B\in \F$. We define the \emph{up-closure} of $\F$ to be
$$
\F^\uparrow=\{B\subseteq[n]: A\subseteq B \text{ for some } A\in\F\},
$$
i.e., it is the minimal increasing subfamily of $2^{[n]}$ that contains $\F$.

We will often identify $\{0,1\}^n$ with $2^{[n]}$ via the correspondence $(x_1,\dots,x_n)\leftrightarrow\{i\in[n]:x_i=1\}$. We write $\mathbbm{1}_{\F}$ for the \emph{indicator function} of $\F$, i.e., the Boolean function
$$
\mathbbm{1}_{\F}:\{0,1\}^n\to\{0,1\},\quad \mathbbm{1}_{\F}(x)=\begin{cases}
1, &\text{if $x\in \F$},\\
0, &\text{if $x\notin \F$}.
\end{cases}
$$

Let $0\le p\le1$. By identifying $\{0,1\}^n$ with $2^{[n]}$ as above, the
$p$-biased measure on $2^{[n]}$ can alternatively be viewed as the product
measure on $\{0,1\}^n$ under which the coordinates are independent with
    $\mathbb P(x_i=1)=p$ and  $\mathbb P(x_i=0)=1-p.$
Equivalently, with the convention that $0^0=1$,
\[
\mu_p(A)=\sum_{x\in A}p^{|\{i\in[n]:x_i=1\}|}(1-p)^{|\{i\in[n]:x_i=0\}|},
\qquad
\forall A\subseteq\{0,1\}^n.
\]

\subsection{Biased Fourier analysis}

Here we summarize some notation and basic properties of Fourier analysis on $\{0,1\}^n$. We fix
$p\in(0,1)$ and consider $\{0,1\}^n$ to be equipped with the $p$-biased measure $\mu_p$. For each $i\in[n]$ we define $\chi_i:\{0,1\}^n\to\mathbb{R}$ by $\chi_i(x)=\frac{x_i-p}{\sqrt{p(1-p)}}$ (so $\chi_i$ has mean $0$ and variance $1$). 
We use the orthonormal Fourier basis $\{\chi_S\}_{S\subseteq[n]}$ of $L^2(\{0,1\}^n,\mu_p)$, where each $\chi_S:=\prod_{i\in S}\chi_i$. Any $f:\{0,1\}^n \to \mathbb{R}$ has a unique expression $f=\sum_{S\subseteq[n]}\hat{f}(S)\chi_S$ where $\{\hat{f}(S)\}_{S\subseteq[n]}$ are the $p$-biased Fourier coefficients of $f$. Orthonormality gives the Plancherel identity $\tup{f,g}=\sum_{S\subseteq[n]}\hat{f}(S)\hat{g}(S)$. In particular, we have the Parseval identity $\mathbb{E}[f^2]=\|f\|_2^2=\tup{f,f}=\sum_{S\subseteq[n]}\hat{f}(S)^2$.

We shall use the following monotone consequence of the biased FKN theorem.  The original
FKN theorem of Friedgut, Kalai and Naor~\cite{FKN2002} was proved for the uniform measure, and Kindler and Safra
\cite{kindler2002property,kindler2002noise} extended it to the $p$-biased measure for arbitrary
fixed $p\in(0,1)$.  In its usual form, the conclusion allows Boolean functions depending on one
coordinate; in set notation, these are stars 
$    \mathcal S_i=\{A\subseteq[n]:i\in A\}$
and anti-stars
$    \mathcal S_i^0=\{A\subseteq[n]:i\notin A\}.$
For increasing families, the anti-star alternative $\mathcal S_i^0$ can be absorbed into the
constant-family alternatives.  Indeed, if $\F$ is increasing and $\mu_p(\F\Delta\mathcal S_i^0)=\delta,$
then writing
\[
    \F^0=\{A\subseteq[n]\setminus\{i\}:A\in\F\},
    \qquad
    \F^1=\{A\subseteq[n]\setminus\{i\}:A\cup\{i\}\in\F\},
\]
we have $\F^0\subseteq\F^1$, and therefore
$    \delta=(1-p)(1-\mu_p(\F^0))+p\mu_p(\F^1)
    \ge \min\{p,1-p\}.$
Thus, if the standard FKN theorem returns an anti-star, then either the error is already bounded
below by a positive constant depending only on $p$, in which case one of the constant families
$\varnothing,2^{[n]}$ is close enough after increasing the constant, or else the anti-star case is
impossible.  Hence the standard biased FKN theorem yields the following form.

\begin{theorem}[Biased FKN]\label{lem:biased-FKN}
Let $0<p<1$. There exists a constant $C_{\ref{lem:biased-FKN}}(p)>0$ such that the following holds for every $n\ge 1$. If $\F\subseteq 2^{[n]}$ is increasing and satisfies
$$\sum_{|S|\ge 2}\widehat{1_{\F}}(S)^2\le \varepsilon,$$
then there exists $\mathcal H\in\{\varnothing,2^{[n]},\mathcal{S}_1,\dots,\mathcal{S}_n\}$ such that $\mu_p(\F\Delta \mathcal{H})\le C_{\ref{lem:biased-FKN}}(p)\varepsilon$.
\end{theorem}

\subsection{Stability form of the AM--GM inequality}
We shall use the following standard stability form of the AM--GM inequality due to Cartwright and Field \cite{CF1978AM-GM}.
\begin{theorem}[\cite{CF1978AM-GM}]\label{thm:stability-AM-GM}
    Let $r\ge 2$ and set $p_*:=\frac{r-1}{r}$. Then for every $y_1,\dots,y_r\in [0,r-1]$ satisfying $\sum_{i=1}^r y_i=r-1$, we have $$p_*-\prod_{i=1}^r y_i^{\frac{1}{r}}\ge\frac{1}{2r(r-1)}\sum_{i=1}^r (y_i-p_*)^2.$$
\end{theorem}

We shall need the following variant in which the total sum is allowed to be smaller than $r-1$. This is the form needed after applying the critical-sum lemma (Lemma~\ref{lem:critical-sum}) gives the required upper bound on the sum of the critical-bias measures, and the lemma below then converts near-optimality of the product into closeness to the balanced value $p_*$.

\begin{lemma}\label{lem:agm-stability-r}
Let $r\ge 2$ and set $p_*:=\frac{r-1}{r}$.  For every $x_1,\dots,x_r\ge 0$ satisfying $\sum_{i=1}^r x_i\le r-1$, one has 
$$p_*^r-\prod_{i=1}^r x_i\ge \frac{p_*^{r-1}}{4r(r-1)}\sum_{i=1}^r (x_i-p_*)^2.$$
\end{lemma}

\begin{proof}

Set $s:=r-1-\sum_{i=1}^r x_i\ge 0$ and $y_i:=x_i+\frac{s}{r}$. Then $\sum_{i=1}^r y_i=r-1$ and $x_i\le y_i$ for every $i$. Also, $y_i\ge0$ for every $i$, and since the $y_i$'s are nonnegative and sum to $r-1$, we have $y_i\le r-1$ for every $i$. Hence $\prod_{i=1}^r x_i\le \prod_{i=1}^r y_i$. By Theorem~\ref{thm:stability-AM-GM}, $p_*-\prod_{i=1}^r y_i^{\frac{1}{r}}\ge \frac{1}{2r(r-1)}\sum_{i=1}^{r}(y_i-p_*)^2$. If $z:=\prod_{i=1}^r y_i^{\frac{1}{r}}$, then $0\le z\le p_*$ by AM--GM, and so $p_*^r-z^r\ge p_*^{r-1}(p_*-z)$. Thus $p_*^r-\prod_{i=1}^r y_i\ge p_*^{r-1}(p_*-\prod_{i=1}^r y_i^{\frac{1}{r}})$, and we have 
    \begin{equation}\label{eq:agm-tangent-r}
        p_*^r-\prod_{i=1}^r x_i\ge p_*^r-\prod_{i=1}^r y_i\ge \frac{p_*^{r-1}}{2r(r-1)}\sum_{i=1}^r (y_i-p_*)^2.
    \end{equation}
    Since $x_i-p_*=(y_i-p_*)-\frac{s}{r}$ and $\sum_{i=1}^r (y_i-p_*)=0$, we have
    \begin{equation}\label{eq:displament-y-x}
        \sum_{i=1}^r (x_i-p_*)^2=\sum_{i=1}^r (y_i-p_*)^2+\frac{s^2}{r}.
    \end{equation}
By AM--GM,
$\prod_{i=1}^r x_i\le\left(\frac{r-1-s}{r}\right)^r=\left(p_*-\frac{s}{r}\right)^r.$
Consequently, 
\begin{equation}\label{eq:agm-normal-r}
p_*^r-\prod_{i=1}^r x_i\ge p_*^r-\left(p_*-\frac{s}{r}\right)^r=p_*^r\left(1-\left(1-\frac{s}{r-1}\right)^r\right)\ge p_*^r\cdot \frac{s^2}{(r-1)^2},
\end{equation}
where the last inequality follows from $r\ge 2$ and $0\le s\le r-1$. 
Combining \eqref{eq:agm-tangent-r}, \eqref{eq:displament-y-x} and \eqref{eq:agm-normal-r} gives
$$p_*^r-\prod_{i=1}^r x_i\ge \frac{1}{2}\min\left\{\frac{p_*^{r-1}}{2r(r-1)},r\cdot \frac{p_*^r}{(r-1)^2}\right\}\left(\sum_{i=1}^r (y_i-p_*)^2+\frac{s^2}{r}\right)=\frac{p_*^{r-1}}{4r(r-1)}\sum_{i=1}^r (x_i-p_*)^2.$$
This completes the proof. 
\end{proof}

\section{Proof of Theorem~\ref{thm:main}}\label{sec:extremal}

\subsection{Critical-sum lemma}
We begin with the critical-bias sum inequality that drives the later extremal arguments.
\begin{lemma}\label{lem:critical-sum}
If $r\ge 2$ and $\F_1,\dots,\F_r\subseteq 2^{[n]}$ are $r$-cross-intersecting, then
\[
\sum_{i=1}^r \mu_{(r-1)/r}(\F_i)\le r-1.
\]
\end{lemma}

\begin{proof}
For each $j\in[n]$, independently choose a random index $R_j\in[r]$ uniformly from $\{1,2,\dots,r\}$. Define random sets $X_1,\dots,X_r\subseteq[n]$ by $j\in X_i$ if and only if $R_j\neq i$. In words, at coordinate $j$ exactly one of the $r$ sets $X_1,\dots,X_r$ misses $j$, namely $X_{R_j}$, while all the others contain $j$.

Fix $i\in[r]$. For each $j\in[n]$, we have
\[
\PP(j\in X_i)=\PP(R_j\neq i)=\frac{r-1}{r}.
\]
Since the choices of the $R_j$ are independent across $j$, each $X_i$ is distributed according to $\mu_{(r-1)/r}$. Therefore,
\[
\PP(X_i\in \F_i)=\mu_{(r-1)/r}(\F_i),\quad \forall i\in[r].
\]

On the other hand, for every $j\in[n]$ there is exactly one index $i$ with $j\notin X_i$. Hence no element of $[n]$ belongs to all of $X_1,\dots,X_r$, so $\cap_{i=1}^r X_i=\varnothing$ with probability $1$. Note that events $\{X_i\in \F_i\}$ cannot all occur simultaneously, since the families are $r$-cross-intersecting. Thus we have
\[
\sum_{i=1}^r \mathbbm{1}_{\{X_i\in \F_i\}}\le r-1.
\]
Taking expectations finishes the proof.
\end{proof}

\subsection{The comparison lemma}
We next record the monotonicity statement that transports critical-bias estimates to smaller biases.

\begin{lemma}\label{lem:comparison}
Let $0<p\le q<1$. If $\mathcal{F}\subseteq 2^{[n]}$ is increasing, then
\[
\mu_p(\mathcal{F})\le \mu_q(\mathcal{F})^{\frac{\log p}{\log q}}.
\]
Equality holds for increasing subcubes. 
\end{lemma}

\begin{proof}
We argue by induction on $n$. 
The case $n=0$ is immediate. 
Let $n\ge 1$, and suppose the statement has been proved in dimension $n-1$. For an increasing family  $\mathcal{F}\subseteq 2^{[n]}$, define its two sections with respect to coordinate $n$ by 
$$\mathcal{F}^0:=\{A\subseteq [n-1]: A\in\mathcal{F}\}, \quad \mathcal{F}^1:=\{A\subseteq [n-1]: A\cup\{n\}\in\mathcal{F}\}. $$
Both $\mathcal{F}^0$ and $\mathcal{F}^1$ are increasing families on $2^{[n-1]}$ and $\mathcal{F}^0\subseteq \mathcal{F}^1$. Set $\alpha:=\frac{\log p}{\log q}\ge 1$. By the induction hypothesis, 
$$\mu_p(\F^0)\le \mu_q(\F^0)^\alpha,\quad \mu_p(\F^1)\le \mu_q(\F^1)^\alpha.$$ 
Therefore,
\[
\mu_p(\F)=(1-p)\cdot\mu_p(\F^0)+p\cdot\mu_p(\F^1)\le (1-p)\cdot\mu_q(\F^0)^\alpha+p\cdot\mu_q(\F^1)^\alpha.
\]
Thus it remains to prove that, for all $0\le x\le y$,
\begin{equation}\label{ineq:main two-point}
 (1-p)x^\alpha+p y^\alpha\le \left((1-q)x+q y\right)^\alpha.
\end{equation}

If $y=0$, then also $x=0$, and both sides are $0$. Assume $y>0$ and put  $t:=\frac{x}{y}\in[0,1]$. After dividing by $y^\alpha$, it is enough to show that
$(1-p)t^\alpha+p\le \bigl((1-q)t+q\bigr)^\alpha.$
The endpoint cases $t=0$ and $t=1$ give equality, since $p=q^{\alpha}$. Thus assume that $0<t<1$ and write $t=\frac{1}{1+u}$ with $u=\frac{1-t}{t}>0$. Then the desired inequality is equivalent to
\[
(1+qu)^\alpha-1\ge q^\alpha\bigl((1+u)^\alpha-1\bigr).
\]
Now define
$\phi(u):=\frac{(1+u)^\alpha-1}{u^\alpha}.$
A direct differentiation gives
\[
\phi'(u)
=\frac{\alpha}{u^{\alpha+1}}\Bigl(u(1+u)^{\alpha-1}-\bigl((1+u)^\alpha-1\bigr)\Bigr)
=\frac{\alpha\bigl(1-(1+u)^{\alpha-1}\bigr)}{u^{\alpha+1}}\le 0,
\]
because $\alpha\ge 1$. Hence $\phi$ is decreasing on $(0,\infty)$. Since $0<q<1$, we have $qu\le u$, and therefore $\phi(qu)\ge \phi(u)$. This means
$\frac{(1+qu)^\alpha-1}{(qu)^\alpha}
\ge
\frac{(1+u)^\alpha-1}{u^\alpha}.$
Multiplying both sides by $(qu)^\alpha$ yields
\[
(1+qu)^\alpha-1\ge q^\alpha\bigl((1+u)^\alpha-1\bigr).
\]
We now apply \eqref{ineq:main two-point} with
    $x=\mu_q(\mathcal F^0)$ and $y=\mu_q(\mathcal F^1).$
This choice satisfies $0\le x\le y$, because $\mathcal F^0\subseteq\mathcal F^1$. Therefore
\[
\mu_p(\F)
\le
(1-p)x^\alpha+py^\alpha
\le
((1-q)x+qy)^\alpha
=
\mu_q(\F)^\alpha,
\]
which completes the induction.

Finally, if
$\mathcal{F}=\{A\subseteq[n]:\{1,\ldots,m\}\subseteq A\},$
then $\mu_p(\F)=p^m$ and $\mu_q(\F)=q^m$, so equality holds.
\end{proof}

\subsection{Proof of the product bound}
We now turn the critical sum inequalities into the stated product bounds. Theorem~\ref{thm:main} is the point where the two core lemmas enter directly.

\begin{proof}[Proof of Theorem~\ref{thm:main}]
First consider the case $p=0$. For each $i$, we have $\mu_0(\F_i)=1$ if and only if $\varnothing\in \F_i$, and otherwise $\mu_0(\F_i)=0$. Therefore, if some family $\F_i$ does not contain $\varnothing$, then the product is immediately $0$. If all $r$ families contained $\varnothing$, then taking $A_i=\varnothing\in \F_i$ for every $i$ would contradict the $r$-cross-intersecting condition. So the product is always $0$ when $p=0$.

Now assume $0<p\le \frac{r-1}{r}$. Let $q:=\frac{r-1}{r}$ and $\alpha:=\frac{\log p}{\log q}$. Then $\alpha\ge 1$ and $p=q^\alpha$. For each $i\in[r]$, let $\F_i^{\uparrow}$ be the up-closure of $\F_i$. Then $\F_i^{\uparrow}$ is increasing and $\F_i\subseteq \F_i^{\uparrow}$, so $\mu_p(\F_i)\le \mu_p(\F_i^{\uparrow})$ for any $i\in[r]$. Moreover, the families $\F_1^{\uparrow},\dots,\F_r^{\uparrow}$ are still $r$-cross-intersecting.  Hence
\begin{equation*}
    \begin{split}
  \prod_{i=1}^r \mu_p(\F_i)&\le \prod_{i=1}^r \mu_p(\F_i^{\uparrow})\overset{{\rm Lemma}~\ref{lem:comparison}}{\le} \prod_{i=1}^r\mu_q(\F_i^{\uparrow})^{\alpha}=\left(\prod_{i=1}^r \mu_q(\F_i^{\uparrow})\right)^{\alpha}\\
  &\overset{{\rm AM-GM}}{\le} \left(\frac{1}{r}\sum_{i=1}^r \mu_q(\F_i^{\uparrow})\right)^{r\cdot\alpha}\overset{{\rm Lemma}~\ref{lem:critical-sum}}{\le} \left(\frac{r-1}{r}\right)^{r\alpha}=\left(q^{\frac{\log p}{\log q}}\right)^r=p^r.    
    \end{split}
\end{equation*}
This proves the theorem.
\end{proof}

\section{Proof of Theorem~\ref{thm:r-cross-product-stability}}\label{sec:stability}

\subsection{Dual families}
For a family $\A\subseteq2^{[n]}$, define its dual family by
\[
    \A^*:=\{B\subseteq[n]:[n]\setminus B\notin\A\}.
\]
When several families $\A_1,\dots,\A_r$ are under consideration, we write
    $\B_i:=\A_i^*
    =
    \{B\subseteq[n]:[n]\setminus B\notin\A_i\}.$
Studying such dual families is natural in intersection problems.  If $\A$ is increasing, then
$\A^*$ is exactly the family of all hitting sets, or transversals, for $\A$: one has
$B\in\A^*$ if and only if
$B\cap A\ne\varnothing$
for every $A\in\A$.
Thus an increasing family $\A$ is intersecting if and only if $\A\subseteq\A^*$; maximal
intersecting up-sets are precisely the self-dual increasing families $\A=\A^*$. This self-dual viewpoint is classical in the enumeration of maximal intersecting families; see for example~\cite{BDDLS2015}. Dual pairs also appear naturally in probabilistic combinatorics, for instance in the tightness examples for Talagrand-type correlation inequalities studied by Kalai, Keller and Mossel \cite{KKM2016}.

For two increasing families $\A_1,\A_2\subseteq 2^{[n]}$, they are cross-intersecting if and only if $\A_1\subseteq \B_2$. Consequently $\mu_p(\A_1)\le \mu_p(\B_2)=1-\mu_{1-p}(\A_2)$, and in particular $\mu_{1/2}(\A_1)+\mu_{1/2}(\A_2)\le 1$.  This kind of dual comparison is a standard first step in several stability and threshold arguments for increasing families; see, for instance, \cite{EllisKellerLifshitzJEMS2019}. In the present $r$-cross-intersecting setting, however, the condition is genuinely higher-order and cannot be captured just by pairwise inclusions.  The useful replacement is the ordered-partition covering property in Lemma~\ref{lem:dual-family} below.

\begin{lemma}\label{lem:dual-family}
Let $r\ge 2$. Suppose that $\A_1,\dots,\A_r\subseteq 2^{[n]}$ are increasing and $r$-cross-intersecting, and let $\B_i:=\A_i^*$, $i\in[r]$, be their dual families. Then the following properties hold.
\begin{itemize}
    \item[(i)] Every $\B_i$ is an increasing family.
    \item[(ii)] $\mu_{1/r}(\B_i)=1-\mu_{(r-1)/r}(\A_i)$ for every $i\in[r]$.
    \item[(iii)] For every ordered partition $P_1,\dots,P_r$ of $[n]$ into possibly empty parts, there exists some $i\in[r]$ such that $P_i\in\B_i$.
\end{itemize}
\end{lemma}

\begin{proof}
We first prove (i). Fix $i\in[r]$. Let $B\in\B_i$ and let $B'\supseteq B$. By the definition of $\B_i$, we have $[n]\setminus B\notin \A_i$. Since $B'\supseteq B$, we have $[n]\setminus B'\subseteq [n]\setminus B$. If $[n]\setminus B'\in\A_i$, then, as $\A_i$ is increasing, it would follow that $[n]\setminus B\in\A_i$, a contradiction. Hence $[n]\setminus B'\notin\A_i$, and therefore $B'\in\B_i$. Thus $\B_i$ is increasing.

Next we prove (ii). Let $X\sim\mu_{1/r}$ be a random subset of $[n]$. Then $[n]\setminus X\sim\mu_{(r-1)/r}$. Hence, by the definition of $\B_i$,
$$
\mu_{1/r}(\B_i)=\mathbb{P}_{X\sim\mu_{1/r}}(X\in \B_i)=\mathbb{P}_{X\sim\mu_{1/r}}([n]\setminus X\notin \A_i)=1-\mu_{(r-1)/r}(\A_i).
$$

Finally we prove (iii). Let $P_1,\dots,P_r$ be an ordered partition of $[n]$ into possibly empty parts.
Suppose for contradiction that $P_i\notin\B_i$ for every $i\in[r]$. Then, by the definition of $\B_i$, we have $[n]\setminus P_i\in\A_i$ for every $i\in[r]$. But since $P_1,\dots,P_r$ form a partition of $[n]$, we have $\bigcap_{i=1}^r([n]\setminus P_i)=\varnothing$. This contradicts the assumption that $\A_1,\dots,\A_r$ are $r$-cross-intersecting. Therefore $P_i\in\B_i$ for some $i\in[r]$, as desired.
\end{proof}

Lemma~\ref{lem:dual-family}(iii) is exactly what is needed for the Fourier argument.  If $x\in[r]^n$ is chosen uniformly and $C_i(x):=\{a\in[n]:x_a=i\}$, then $C_1(x),\dots,C_r(x)$ form a random ordered partition of $[n]$ into possibly empty parts, and each $C_i(x)$ has distribution $\mu_{1/r}$.
Lemma~\ref{lem:dual-family}(iii) gives $\sum_i \mathbbm{1}_{\B_i}(C_i(x))\ge 1$ for every $x\in[r]^n$, while near equality in the critical sum inequality gives $\sum_i\mu_{1/r}(\B_i)\le 1+\eta$.  Hence $\sum_i \mathbbm{1}_{\B_i}(C_i(x))$ is close to the constant function $1$.  This is the starting point for proving that the high-degree Fourier mass of the dual families is small, and then applying the biased FKN theorem to obtain the structural alternatives in Theorem~\ref{lem:q-stability-r}.

\subsection{Proof of Theorem~\ref{lem:q-stability-r}}

\begin{proof}[Proof of Theorem~\ref{lem:q-stability-r}]
Set 
$$\eta_{\ref{lem:q-stability-r}}(r):=\min\left\{\frac{r(r-2)}{3C_{\ref{lem:biased-FKN}}(\frac{1}{r})(r-1)^3+3r(r-2)},\frac{r-2}{2C_{\ref{lem:biased-FKN}}(\frac{1}{r})r(r-1)^3},\frac{r-2}{2r^rC_{\ref{lem:biased-FKN}}(\frac{1}{r})(r-1)^3}\right\}$$ and 
$$C_{\ref{lem:q-stability-r}}(r):=\frac{C_{\ref{lem:biased-FKN}}(\frac{1}{r})(r-1)^3+r(r-2)}{r(r-2)}.$$
For each $i\in[r]$, the up-closure $\A_i^{\uparrow}$ is increasing and $\A_i\subseteq\A_i^{\uparrow}$. Moreover, $\A_1^{\uparrow},\dots,\A_r^{\uparrow}$ are still $r$-cross-intersecting. 
Set
$\xi:=r-1-\sum_{i=1}^r\mu_{p_*}(\A_i^{\uparrow})$ and
$\delta:=\sum_{i=1}^r\mu_{p_*}(\A_i^{\uparrow}\setminus\A_i).$ By Lemma~\ref{lem:critical-sum}, applied to $\A_1^{\uparrow},\dots,\A_r^{\uparrow}$, and the assumption, we have $0\le \xi\le \eta$. Also,
$$\delta=\sum_{i=1}^r\mu_{p_*}(\A_i^{\uparrow})-\sum_{i=1}^r\mu_{p_*}(\A_i)\le (r-1-\xi)-(r-1-\eta)=\eta-\xi.$$
For each $i\in[r]$, let $\B_i$ be the dual family of $\A_i^{\uparrow}$, namely $\B_i:=\{B\subseteq[n]:[n]\setminus B\notin \A_i^{\uparrow}\}.$ By Lemma~\ref{lem:dual-family}, applied to $\A_1^{\uparrow},\dots,\A_r^{\uparrow}$, each $\B_i$ is increasing and 
$\mu_{1/r}(\B_i)=1-\mu_{p_*}(\A_i^{\uparrow})$ for every $i\in[r]$. Thus $\sum_{i=1}^r\mu_{1/r}(\B_i)=1+\xi$.

Let $x=(x_1,\dots,x_n)$ be chosen uniformly from $[r]^n$, and for each $i\in[r]$ define $C_i(x):=\{a\in[n]:x_a=i\}.$ Thus $C_1(x),\dots,C_r(x)$ form an ordered partition of $[n]$ into possibly empty parts. Define
$g_i(x):=\mathbbm{1}_{\B_i}(C_i(x))$. For fixed $i$, the random set $C_i(x)$ has distribution $\mu_{1/r}$, since each coordinate belongs to $C_i(x)$ independently with probability $1/r$. Hence
$$
\mathbb{E}_{x\sim[r]^n}[g_i(x)]=\mathbb{P}_{x\sim[r]^n}[C_i(x)\in \B_i]=\mu_{1/r}(\B_i).
$$
Moreover, by Lemma~\ref{lem:dual-family} (iii), for every $x\in[r]^n$ there exists some $i\in[r]$ such that $C_i(x)\in\B_i$. Therefore $\sum_{i=1}^r g_i(x)\ge 1$ for every $x\in[r]^n$. Consequently,
$$
\mathbb{E}_{x\sim[r]^n}\left[\sum_{i=1}^r g_i(x)-1\right]=\sum_{i=1}^r\mu_{1/r}(\B_i)-1=\xi.
$$
Since $0\le \sum_{i=1}^r g_i-1\le r-1$, we also have
\begin{equation}\label{eq:L2-defect}
\mathbb{E}_{x\sim[r]^n}\left[\left(\sum_{i=1}^r g_i(x)-1\right)^2\right]\le (r-1)\xi.
\end{equation}

\begin{claim}
For $y\in[r]$, we define a function $u_i(y)=\chi(\mathbbm{1}_{\{y=i\}})=\frac{\mathbbm{1}_{\{y=i\}}-\frac{1}{r}}{\sqrt{\frac{1}{r}(1-\frac{1}{r})}}$ for any $i\in[r]$. Then $\mathbb{E}_{y\sim[r]}[u_i(y)]=0$, $\mathbb{E}_{y\sim[r]}[u_i(y)^2]=1$, and for $i\ne j$,
$$\mathbb{E}_{y\sim[r]}[u_i(y)\cdot u_j(y)]=-\frac{1}{r-1}.$$
\end{claim}
These follow from the fact that $\mathbbm{1}_{\{y=i\}}\sim\mu_{1/r}$ when $y\sim [r]$. For every $S\subseteq[n]$, write $u_i^S(x):=\prod_{a\in S}u_i(x_a)=\prod_{a\in S}\chi_a(\mathbbm{1}_{\{x_a=i\}}),\forall x\in[r]^n$, with the convention $u_i^\varnothing=1$. We shall use the following orthogonality observation.
\begin{claim}
For any $S,T\subseteq[n]$ with $S\neq T$ and any $i,j\in[r]$, $\mathbb{E}_{x\sim[r]^n}[u_i^S(x)\cdot u_j^T(x)]=0$.
\end{claim}
Indeed, choose $a\in S\Delta T$, say $a\in S\setminus T$. Then the product $u_i^S(x)u_j^T(x)$ contains the factor $u_i(x_a)$ and no other factor depending on $x_a$. Since $\mathbb{E}_{x_a\sim[r]}[u_i(x_a)]=0$ and the coordinates are independent, the expectation is $0$. In other words,
\begin{equation*}
\underset{x\sim[r]^n}{\mathbb{E}}\left[u_i^S(x)\cdot u_j^T(x)\right]=\underset{x_a\sim[r]}{\mathbb{E}}\left[\frac{\mathbbm{1}_{\{x_a=i\}}-\frac{1}{r}}{\sqrt{\frac{1}{r}(1-\frac{1}{r})}}\right]\cdot\underset{x_{[n]\setminus\{a\}}\sim[r]^{n-1}}{\mathbb{E}}\left[u_i^{S\setminus\{a\}}(x)\cdot u_j^T(x)\right]=0.
\end{equation*}

Now using $1/r$-biased Fourier expansion of $\mathbbm{1}_{\B_i}$, we have 
$$
g_i(x)=\mathbbm{1}_{\B_i}(C_i(x))=\sum_{S\subseteq[n]}\widehat{\mathbbm{1}_{\B_i}}(S)u_i^S(x),\quad \forall i\in[r].
$$
Therefore $\sum_{i=1}^r g_i(x)-1=(\sum_{i=1}^r\widehat{\mathbbm{1}_{\B_i}}(\varnothing)-1)+\sum_{S\neq \varnothing}\sum_{i=1}^r\widehat{\mathbbm{1}_{\B_i}}(S)u_i^S(x)$. By the orthogonality observation for different sets $S$, we get
\begin{equation*}
    \begin{split}
 \underset{x\sim[r]^n}{\mathbb{E}}\left[\left(\sum_{i=1}^r g_i(x)-1\right)^2\right]&=\left(\sum_{i=1}^r\widehat{\mathbbm{1}_{\B_i}}(\varnothing)-1\right)^2+\sum\limits_{S\neq\varnothing}\underset{x\sim[r]^n}{\mathbb{E}}\left[\left(\sum_{i=1}^r\widehat{\mathbbm{1}_{\B_i}}(S)u_i^S(x)\right)^2\right]\\
 &\ge \sum\limits_{|S|\ge2}\underset{x\sim[r]^n}{\mathbb{E}}\left[\left(\sum_{i=1}^r\widehat{\mathbbm{1}_{\B_i}}(S)u_i^S(x)\right)^2\right].
    \end{split}
\end{equation*}
We now estimate each block with $|S|\ge2$. Let $S\subseteq[n]$ with $|S|=s\ge 2$. The Gram matrix of the vectors
$u_1^S,\dots,u_r^S$ has diagonal entries $1$ and off-diagonal entries
$\left(-\frac{1}{r-1}\right)^s$. Indeed, this follows by multiplying the corresponding one-coordinate inner products over the $s$ coordinates in $S$. If $\rho_s:=\left(-\frac{1}{r-1}\right)^s$, then this Gram matrix is
$$(1-\rho_s)I+\rho_s J,$$
where $J$ is the all-one matrix. Its eigenvalues are $1-\rho_s$, with multiplicity $r-1$, and $1+(r-1)\rho_s$, with multiplicity $1$. Since $s\ge2$, both eigenvalues are at least $1-\frac{1}{(r-1)^2}=\frac{r(r-2)}{(r-1)^2}>0$.

Combining with~\eqref{eq:L2-defect}, we get
$$\frac{r(r-2)}{(r-1)^2}\cdot\sum_{i=1}^r\sum_{|S|\ge 2}\widehat{\mathbbm{1}_{\B_i}}(S)^2\le (r-1)\xi.$$
Thus, for every $i\in[r]$,
$$\sum_{|S|\ge 2}\widehat{\mathbbm{1}_{\B_i}}(S)^2\le \frac{(r-1)^3}{r(r-2)}\xi.$$

By~\cref{lem:biased-FKN}, applied with bias $\frac{1}{r}$ and
$\varepsilon=\frac{(r-1)^3}{r(r-2)}\xi$, for every $i\in[r]$ there is $\mathcal{H}_i\in\{\varnothing,2^{[n]},\mathcal{S}_1,\dots,\mathcal{S}_n\}$
such that
\begin{equation}\label{eq:Di-close-to-extremal}
\mu_{1/r}(\B_i\Delta \mathcal H_i)\le \frac{C_{\ref{lem:biased-FKN}}(\frac{1}{r})(r-1)^3}{r(r-2)}\xi.
\end{equation}
\emph{Case 1. $\mathcal H_s=2^{[n]}$ for some $s\in[r]$}. Then \eqref{eq:Di-close-to-extremal} gives $\mu_{1/r}(\B_s)\ge 1-\frac{C_{\ref{lem:biased-FKN}}(\frac{1}{r})(r-1)^3}{r(r-2)}\xi$. Since $\sum_i\mu_{1/r}(\B_i)=1+\xi$, it follows that, for every $i\ne s$, $\mu_{1/r}(\B_i)\le (\frac{C_{\ref{lem:biased-FKN}}(\frac{1}{r})(r-1)^3}{r(r-2)}+1)\xi$. Therefore,
$$\mu_{p_*}(\A_s^{\uparrow})=1-\mu_{1/r}(\B_s)\le \frac{C_{\ref{lem:biased-FKN}}(\frac{1}{r})(r-1)^3}{r(r-2)}\xi,$$
and, for every $i\ne s$,
$$\mu_{p_*}(\A_i^{\uparrow})=1-\mu_{1/r}(\B_i)\ge 1-\left(1+\frac{C_{\ref{lem:biased-FKN}}(\frac{1}{r})(r-1)^3}{r(r-2)}\right)\xi.$$
Since $\A_s\subseteq\A_s^{\uparrow}$, we have
$$\mu_{p_*}(\A_s)\le \mu_{p_*}(\A_s^{\uparrow})\le \frac{C_{\ref{lem:biased-FKN}}(\frac{1}{r})(r-1)^3}{r(r-2)}\eta\le C_{\ref{lem:q-stability-r}}(r)\eta.$$
For every $i\ne s$, using $\mu_{p_*}(\A_i^{\uparrow}\setminus\A_i)\le\delta\le\eta-\xi$, we have
\begin{equation*}
\begin{split}
\mu_{p_*}(\A_i)&=\mu_{p_*}(\A_i^{\uparrow})-\mu_{p_*}(\A_i^{\uparrow}\setminus\A_i)\ge 1-\left(1+\frac{C_{\ref{lem:biased-FKN}}(\frac{1}{r})(r-1)^3}{r(r-2)}\right)\xi-(\eta-\xi)\\
&=1-\frac{C_{\ref{lem:biased-FKN}}(\frac{1}{r})(r-1)^3}{r(r-2)}\xi-\eta\ge 1-C_{\ref{lem:q-stability-r}}(r)\eta.
\end{split}
\end{equation*}
Since $C_{\ref{lem:q-stability-r}}(r)\eta\le C_{\ref{lem:q-stability-r}}(r)\eta_{\ref{lem:q-stability-r}}(r)\le \frac13$, the index $s$ is unique. This gives alternative (ii).

\emph{Case 2. No $\mathcal H_i$ is $2^{[n]}$}. We claim that no $\mathcal H_i$ is $\varnothing$. Indeed, if $\mathcal H_s=\varnothing$, then $\mu_{1/r}(\B_s)\le \frac{C_{\ref{lem:biased-FKN}}(\frac{1}{r})(r-1)^3}{r(r-2)}\xi$. Since every other $\mathcal H_i$ is either $\varnothing$ or a $1$-star, we have $\mu_{1/r}(\B_i)\le \frac1r+\frac{C_{\ref{lem:biased-FKN}}(\frac{1}{r})(r-1)^3}{r(r-2)}\xi$ for every $i\ne s$. Therefore
\[
\sum_{i=1}^r\mu_{1/r}(\B_i)\le \frac{r-1}{r}+\frac{C_{\ref{lem:biased-FKN}}(\frac{1}{r})(r-1)^3}{r-2}\xi
<1.
\]
This contradicts $\sum_i\mu_{1/r}(\B_i)=1+\xi\ge 1$. Thus every $\mathcal H_i$ is a $1$-star.

Write $\mathcal H_i=S_{j_i}'$, where $S_{j_i}':=\{B\subseteq[n]:j_i\in B\}$. We claim that $j_1=\cdots=j_r$. Suppose not. For each coordinate $a\in[n]$, let
$I_a:=\{i\in[r]:j_i=a\}.$ Thus $I_a$ is the set of indices whose approximating star has center $a$. Since $j_1,\dots,j_r$ are not all equal,  $I_a\ne[r]$ for every $a\in[n]$. Let $J:=\{j_1,\dots,j_r\}$. For every $a\in J$, the set $[r]\setminus I_a$ is nonempty. Hence we may choose a coloring $x\in[r]^n$ such that $x_a\notin I_a$ for $a\in J$. Let $G$ be the set of all such colorings. Since $|J|\le r$ and each restricted coordinate has at least one available color, we have
$$\mathbb P_{x\sim[r]^n}(x\in G)=\prod_{a\in J}\frac{r-|I_a|}{r}\ge r^{-|J|}\ge r^{-r}.$$
Fix $i\in[r]$. Since $i\in I_{j_i}$ and $x_{j_i}\notin I_{j_i}$, we have $x_{j_i}\ne i$. Therefore $j_i\notin C_i(x)$, and so $C_i(x)\notin S_{j_i}'$
for every $i\in[r]$ and every $x\in G$.

On the other hand, for every $x\in[r]^n$, the sets $C_1(x),\dots,C_r(x)$ form an ordered partition of $[n]$ into possibly empty parts. By Lemma~\ref{lem:dual-family} (iii), there exists some $i\in[r]$ such that $C_i(x)\in\B_i$. Thus, for every $x\in G$, there exists some $i\in[r]$ such that $C_i(x)\in\B_i$ and $C_i(x)\notin S_{j_i}'$. Hence
$$G\subseteq \bigcup_{i=1}^r\{x\in[r]^n:C_i(x)\in\B_i\Delta S_{j_i}'\}.$$
Taking probabilities and using that $C_i(x)\sim\mu_{1/r}$, we get
\[
r^{-r}\le \mathbb P_{x\sim[r]^n}(x\in G)\le \sum_{i=1}^r\mu_{1/r}(\B_i\Delta S_{j_i}')\le \frac{C_{\ref{lem:biased-FKN}}(\frac{1}{r})(r-1)^3}{r-2}\eta
<\frac{1}{r^r}.\]
Hence $j_1=\cdots=j_r$, say $j_i=j$ for every $i\in[r]$.

Finally, we first pass back to the up-closures. If $A=[n]\setminus B$, then $B\in\B_i\Leftrightarrow A\notin\A_i^{\uparrow}$, and $B\in \mathcal{S}_j'\Leftrightarrow A\notin \mathcal{S}_j$, where $\mathcal{S}_j=\{A\subseteq[n]:j\in A\}$ is the original $1$-star. Hence membership in $\B_i\Delta \mathcal{S}_j'$ is exactly transformed into membership in $\A_i^{\uparrow}\Delta \mathcal{S}_j$. Since complementation sends \(\mu_{1/r}\) to \(\mu_{p_*}\),
$$\mu_{p_*}(\A_i^{\uparrow}\Delta \mathcal{S}_j)=\mu_{1/r}(\B_i\Delta \mathcal{S}_j')\le \frac{C_{\ref{lem:biased-FKN}}(\frac{1}{r})(r-1)^3}{r(r-2)}\xi$$
for every $i\in[r]$. Since $\A_i\subseteq\A_i^{\uparrow}$, we have
\begin{equation*}
\begin{split}
\mu_{p_*}(\A_i\Delta \mathcal{S}_j)&\le \mu_{p_*}(\A_i^{\uparrow}\Delta \mathcal{S}_j)+\mu_{p_*}(\A_i^{\uparrow}\setminus\A_i)\le \frac{C_{\ref{lem:biased-FKN}}(\frac{1}{r})(r-1)^3}{r(r-2)}\xi+\delta\\
&\le \frac{C_{\ref{lem:biased-FKN}}(\frac{1}{r})(r-1)^3}{r(r-2)}\xi+\eta-\xi\le C_{\ref{lem:q-stability-r}}(r)\eta
\end{split}
\end{equation*}
for every $i\in[r]$. This gives alternative (i), and completes the proof.
\end{proof}

\begin{cor}\label{cor:q-stability-balanced-r}
Let $r\ge 3$, $p_*=\frac{r-1}{r}$, and let $0<\eta<\eta_{\ref{lem:q-stability-r}}(r)$. Let $\A_1,\dots,\A_r\subseteq 2^{[n]}$ be increasing and $r$-cross-intersecting. Suppose that
$\sum_{i=1}^r \mu_{p_*}(\A_i)\ge r-1-\eta$
and $\left|\mu_{p_*}(\A_i)-{p_*}\right|\le \frac1{4r}$ for all $i\in[r]$. Then there exists a $1$-star $\mathcal{S}_j$ such that
$$\mu_{p_*}(\A_i\Delta \mathcal{S}_j)\le C_{\ref{lem:q-stability-r}}(r)\eta,
\quad \forall i\in[r].$$
\end{cor}
\begin{proof}
Applying Theorem~\ref{lem:q-stability-r}, either (i) or (ii) holds. If (i) holds, then we are done. We claim that (ii) is impossible. Indeed, in (ii), after permuting the indices, one has $\mu_{p_*}(\A_1)\le C_{\ref{lem:q-stability-r}}(r)\eta$. On the other hand, the balanced assumption gives $\mu_{p_*}(\A_1)\ge p_* -\frac{1}{4r}=\frac{4r-5}{4r}$. This contradicts the definition of $\eta_{\ref{lem:q-stability-r}}(r)$, since 
$C_{\ref{lem:q-stability-r}}(r)\eta_{\ref{lem:q-stability-r}}(r)<\frac{4r-5}{4r}$.
\end{proof}

\subsection{Stability for $r$-cross-intersecting families}
In this subsection we prove Theorem~\ref{thm:r-cross-product-stability}. The argument mirrors the extremal proof: first we stabilize the AM--GM step at the critical bias, then we identify the near-extremal critical configurations, and finally we transfer the resulting structure back to the original bias.
\begin{proof}[Proof of Theorem~\ref{thm:r-cross-product-stability}]
Set $$\varepsilon_{\ref{thm:r-cross-product-stability}}(r):=\min\left\{\frac{1}{64r^2(r-1)^2},\frac{\eta_{\ref{lem:q-stability-r}}(r)}{2(r-1)},\frac{1}{r^2C_{\ref{lem:q-stability-r}}(r)}\right\}$$ and $$C_{\ref{thm:r-cross-product-stability}}(r):=2+r^2(2r+1)C_{\ref{lem:q-stability-r}}(r).$$
For each $i\in[r]$, let $\G_i$ be the up-closure of $\F_i$. Then $\G_i$ is increasing and $\mu_p(\F_i)\le \mu_p(\G_i)$. Moreover, $\G_1,\dots,\G_r$ are still $r$-cross-intersecting. 

Let $p_*:=\frac{r-1}{r}$ and  $\alpha:=\frac{\log p}{\log p_*}\ge 1$. By Lemma~\ref{lem:comparison}, $\mu_p(\G_i)\le \mu_{p_*}(\G_i)^\alpha$. Therefore
$$(1-\varepsilon)p^r\le\prod_{i=1}^r \mu_p(\F_i) \le \prod_{i=1}^r \mu_p(\G_i)\le \prod_{i=1}^r \mu_{p_*}(\G_i)^\alpha,$$
and hence,
$$\prod_{i=1}^r \mu_{p_*}(\G_i)\ge p_*^r(1-\varepsilon)^{\frac{1}{\alpha}}\ge p_*^r(1-\varepsilon).$$
On the other hand, Lemma~\ref{lem:critical-sum} gives $\sum_{i=1}^r \mu_{p_*}(\G_i)\le r-1$. Applying Lemma~\ref{lem:agm-stability-r}, we get 
$\sum_{i=1}^r (\mu_{p_*}(\G_i)-p_*)^2 \le 4(r-1)^2\varepsilon$.
In particular,
\begin{equation*}
    |\mu_{p_*}(\G_i)-p_*|\le 2(r-1)\sqrt{\varepsilon}\le \frac{1}{4r},\quad \text{ for all } i\in[r],
\end{equation*}
where the last inequality follows from $\varepsilon\le \frac{1}{64r^2(r-1)^2}$. Moreover, by AM--GM,
$$\sum_{i=1}^r \mu_{p_*}(\G_i)\ge r\left(\prod_{i=1}^r \mu_{p_*}(\G_i)\right)^{\frac{1}{r}}\ge r{p_*}(1-\varepsilon)^{\frac{1}{r}}\ge r-1-(r-1)\varepsilon.$$
Since 
$(r-1)\varepsilon<\eta_{\ref{lem:q-stability-r}}(r)$,
Corollary~\ref{cor:q-stability-balanced-r}, applied with $\eta=(r-1)\varepsilon$, gives a coordinate $j\in[n]$ such that
$$\mu_{p_*}(\G_i\Delta \mathcal{S}_j)\le C_{\ref{lem:q-stability-r}}(r)(r-1)\varepsilon,\quad \text{ for all } i\in[r].$$

We next convert this $p_*$-biased closeness into $p$-biased closeness. Relabeling the coordinates if necessary, assume that $\mathcal{S}_j=\mathcal{S}_n=\{A\subseteq[n]:n\in A\}$. For each $i\in[r]$, define $\G_i^0:=\{A\subseteq[n-1]:A\in\G_i\}$ and $\G_i^1:=\{A\subseteq[n-1]:A\cup\{n\}\in\G_i\}$. Since $\G_i$ is increasing, $\G_i^0\subseteq \G_i^1$. Moreover,
$$
\mu_{p_*}(\G_i\Delta \mathcal{S}_n)
=(1-p_*)\mu_{p_*}(\G_i^0)+{p_*}\bigl(1-\mu_{p_*}(\G_i^1)\bigr)
\le C_{\ref{lem:q-stability-r}}(r)(r-1)\varepsilon.$$
In particular,
$$\mu_{p_*}(\G_i^0)\le C_{\ref{lem:q-stability-r}}(r)r(r-1)\varepsilon.$$
Since $\varepsilon\le\varepsilon_{\ref{thm:r-cross-product-stability}}\le \frac{1}{r^2C_{\ref{lem:q-stability-r}}(r) }$, we have $C_{\ref{lem:q-stability-r}}(r)r(r-1)\varepsilon\le p_*$. The slice $\G_i^0$ is increasing, so by Lemma~\ref{lem:comparison}, and using $p=p_*^\alpha$,
\begin{equation}\label{ineq:upper-bound-p-0}
\mu_p(\G_i^0)
\le\mu_{p_*}(\G_i^0)^\alpha=
p\left(\frac{\mu_{p_*}(\G_i^0)}{p_*}\right)^\alpha
\le
p\left(\frac{C_{\ref{lem:q-stability-r}}(r)r(r-1)\varepsilon}{p_*}\right)
\le
pr^2C_{\ref{lem:q-stability-r}}(r)\varepsilon.
\end{equation}
Here we used \(\mu_{p_*}(\G_i^0)\le p_*\) and \(\alpha\ge1\).
Hence,
\begin{equation}\label{ineq:upper-bound-p-biased}
    \mu_p(\mathcal G_i)=(1-p)\mu_p(\mathcal G_i^0)+p\mu_p(\mathcal G_i^1)\le p(1+r^2C_{\ref{lem:q-stability-r}}(r)\varepsilon),\quad \text{ for all }i\in [r].
\end{equation}
Combining \eqref{ineq:upper-bound-p-biased} with $\prod_{i=1}^r \mu_p(\F_i)
\ge(1-\varepsilon)p^r$ gives, for every $i\in[r]$,
\begin{equation}\label{ineq:lower-bound-p-baised}
    \mu_p(\G_i)\ge \mu_p(\F_i)\ge \frac{(1-\varepsilon)p^r}{\left(p(1+r^2C_{\ref{lem:q-stability-r}}(r)\varepsilon)\right)^{r-1}}\ge p\left(1-\left(1+r^2(r-1)C_{\ref{lem:q-stability-r}}(r)\right)\varepsilon\right).
\end{equation}
The last inequality uses \((1+x)^{-(r-1)}\ge 1-(r-1)x\) for \(x\ge0\).
Using \eqref{ineq:upper-bound-p-0} and \eqref{ineq:lower-bound-p-baised}, we obtain
\[
\begin{split}
p\bigl(1-\mu_p(\G_i^1)\bigr)
&=
p+(1-p)\mu_p(\G_i^0)-\mu_p(\G_i)\\
&\le
p\left(1+r^2(r-1)C_{\ref{lem:q-stability-r}}(r)\right)\varepsilon
+
(1-p)\mu_p(\G_i^0)\\
&\le
p\left(1+r^3C_{\ref{lem:q-stability-r}}(r)\right)\varepsilon.
\end{split}
\]
Therefore,
$$\mu_p(\G_i\Delta \mathcal{S}_n)=(1-p)\mu_p(\G_i^0)+p\bigl(1-\mu_p(\G_i^1)\bigr)\le p\left(1+r^2(r+1)C_{\ref{lem:q-stability-r}}(r)\right)\varepsilon.$$

It remains to pass from the up-closures $\G_i$ back to the original families
$\F_i$. By \eqref{ineq:upper-bound-p-biased} and \eqref{ineq:lower-bound-p-baised},
\[
\mu_p(\G_i\setminus\F_i)
=\mu_p(\G_i)-\mu_p(\F_i)\le
p\left(1+r^3C_{\ref{lem:q-stability-r}}(r)\right)\varepsilon.
\]
Finally, for all $i\in[r]$,
\[
\mu_p(\F_i\Delta \mathcal{S}_n)
\le
\mu_p(\G_i\Delta \mathcal{S}_n)+\mu_p(\G_i\setminus\F_i)\le
p\left(2+r^2(2r+1)C_{\ref{lem:q-stability-r}}(r)\right)\varepsilon=
p C_{\ref{thm:r-cross-product-stability}}(r)\varepsilon.
\]
This proves the theorem.
\end{proof}

\section{Concluding Remarks}\label{sec:concluding}

\subsection{A biased-to-uniform consequence}

We record a short consequence of our biased product theorem in the uniform setting.  The
corresponding uniform product bound is already known by the theorem of Frankl and
Tokushige~\cite{FT2011}; the point here is only that the asymptotic form, together with a stability statement, follows directly from Theorems~\ref{thm:main} and
\ref{thm:r-cross-product-stability}.  The argument is the standard biased-to-uniform passage:
take up-closures and use the local LYM inequality at a parameter slightly larger than \(k/n\). For \(j\in[n]\), write
    $\mathcal S_j^{(k)}:=\left\{A\in\binom{[n]}{k}:j\in A\right\}.$

\begin{cor}[Asymptotic uniform product bound]\label{cor:uniform-product}
Let \(r\ge2\), and let \(k=k(n)\) satisfy
    $\frac{k}{n}\to \alpha$ for some $0<\alpha<\frac{r-1}{r}.$
If
$    \A_1,\ldots,\A_r\subseteq\binom{[n]}{k}$
are \(r\)-cross-intersecting, then
\[
    \limsup_{n\to\infty}
    \prod_{i=1}^r\frac{|\A_i|}{\binom nk}
    \le
    \alpha^r.
\]
\end{cor}

\begin{proof}
For each \(i\in[r]\), let \(\G_i:=\A_i^\uparrow\).  Then
\(\G_1,\ldots,\G_r\) are \(r\)-cross-intersecting.  Fix
    $p\in\left(\alpha,\frac{r-1}{r}\right).$
For all sufficiently large \(n\), we have \(k/n<p\).  Put
   $ T_n:=\mathbb P(\operatorname{Bin}(n,p)\ge k).$
By the local LYM inequality, for every \(\ell\ge k\),
$ \frac{|\G_i\cap\binom{[n]}{\ell}|}{\binom n\ell}
    \ge
    \frac{|\A_i|}{\binom nk}.$
Therefore, conditioning on the size of a \(p\)-random set,
\[
    \mu_p(\G_i)
    =
    \sum_{\ell=0}^n
    \mathbb P(\operatorname{Bin}(n,p)=\ell)\cdot
    \frac{|\G_i\cap\binom{[n]}{\ell}|}{\binom n\ell}\ge
    T_n\cdot \frac{|\A_i|}{\binom nk}.
\]

Applying Theorem~\ref{thm:main} to \(\G_1,\ldots,\G_r\), we get
$    T_n^r\prod_{i=1}^r\frac{|\A_i|}{\binom nk}
    \le
    \prod_{i=1}^r\mu_p(\G_i)
    \le
    p^r.$
Since \(\frac{k}{n}\to\alpha<p\), the law of large numbers gives \(T_n\to1\).  Thus
    $\limsup_{n\to\infty}
    \prod_{i=1}^r\frac{|\A_i|}{\binom nk}
    \le p^r.$
Letting \(p\downarrow\alpha\) completes the proof.
\end{proof}

\begin{cor}[Uniform stability]\label{cor:uniform-stability}
Let \(r\ge3\), and let \(0<a<b<\frac{r-1}{r}\).  Let \(k=k(n)\) satisfy
    $a\le \frac{k}{n}\le b.$
Let
$    \A_1,\ldots,\A_r\subseteq\binom{[n]}{k}$
be \(r\)-cross-intersecting. If
$    \prod_{i=1}^r \frac{|\A_i|}{\binom nk}
    \ge
    (1-o(1))\left(\frac{k}{n}\right)^r,$
then there exists \(j\in[n]\) such that, for every \(i\in[r]\),
\[
    \frac{|\A_i\Delta\mathcal S_j^{(k)}|}{\binom{n-1}{k-1}}=o(1).
\]
\end{cor}

\begin{proof}
Write $\alpha_n:=\frac{k}{n}$, $p_n:=\alpha_n+n^{-1/3}$ and $d_i:=\frac{|\A_i|}{\binom nk}.$
For all sufficiently large \(n\), we have \(p_n\le (r-1)/r\).  Let
\(\G_i:=\A_i^\uparrow\).  As in the proof of Corollary~\ref{cor:uniform-product},
$    \mu_{p_n}(\G_i)\ge d_iT_n,$ where
    $T_n:=\mathbb P(\operatorname{Bin}(n,p_n)\ge k).$
Since \(p_n-\alpha_n=n^{-1/3}\), Chernoff's inequality gives \(T_n\to1\).  Hence
\[
    \prod_{i=1}^r\mu_{p_n}(\G_i)
    \ge
    T_n^r\prod_{i=1}^r d_i
    \ge
    (1-o(1))\alpha_n^r
    =
    (1-o(1))p_n^r.
\]
By Theorem~\ref{thm:r-cross-product-stability}, applied at bias \(p_n\), there exists
a \(1\)-star \(\mathcal S_j\subseteq2^{[n]}\) such that, for every \(i\in[r]\),
$   \mu_{p_n}(\G_i\Delta\mathcal S_j)=o(1).$

It remains to return to the \(k\)-th level.  Fix \(i\in[r]\), and set
$\B_i:=\A_i\setminus\mathcal S_j^{(k)}$ and    $b_i:=\frac{|\B_i|}{\binom nk}.$
View \(\B_i\) as a family of \(k\)-subsets of \([n]\setminus\{j\}\), and let
\(\B_i^\uparrow\) be its up-closure inside \([n]\setminus\{j\}\).  Then
\[
    \{X\subseteq[n]:j\notin X,\ X\in\B_i^\uparrow\}
    \subseteq
    \G_i\setminus\mathcal S_j.
\]
By local LYM on \([n]\setminus\{j\}\),
\[
\begin{split}
    \mu_{p_n}(\G_i\setminus\mathcal S_j)
    &\ge
    (1-p_n)
    \frac{|\B_i|}{\binom{n-1}{k}}
    \mathbb P(\operatorname{Bin}(n-1,p_n)\ge k) \\
    &=
    (1-p_n)
    \frac{b_i}{1-\alpha_n}
    \mathbb P(\operatorname{Bin}(n-1,p_n)\ge k).
\end{split}
\]
The binomial tail on the right tends to \(1\), while \(1-p_n\) and \(1-\alpha_n\) are bounded
away from \(0\).  Since
$\mu_{p_n}(\G_i\setminus\mathcal S_j)
    \le
    \mu_{p_n}(\G_i\Delta\mathcal S_j)
    =
    o(1),$
we get
    $b_i=o(1).$
Thus    $d_i\le \alpha_n+o(1)$
    for every $i\in[r].$
The product lower bound then gives, for every \(i\in[r]\),
$    d_i
    \ge
    \frac{(1-o(1))\alpha_n^r}{(\alpha_n+o(1))^{r-1}}
    =
    (1-o(1))\alpha_n,$
using \(\alpha_n\ge a>0\).  Finally, if
$    m_i:=
    \frac{|\mathcal S_j^{(k)}\setminus\A_i|}{\binom nk},$
then
   $ d_i=\alpha_n-m_i+b_i.$
Since \(d_i=(1-o(1))\alpha_n\) and \(b_i=o(1)\), we have \(m_i=o(1)\).  Hence
\[
    \frac{|\A_i\Delta\mathcal S_j^{(k)}|}{\binom nk}
    =
    b_i+m_i
    =
    o(1).
\]
Since $\binom{n-1}{k-1}=\alpha_n\binom nk$ and $\alpha_n\ge a$, the same conclusion holds with normalization by $\binom{n-1}{k-1}$.
\end{proof}

\subsection{A partial unequal-bias consequence}

We next record a simple consequence in the direction of the full Frankl--Tokushige
conjecture.  The argument gives a nontrivial unequal-bias range when the logarithmic biases
are sufficiently balanced.

\begin{cor}[A log-balanced unequal-bias range]\label{cor:log-balanced-unequal}
Let $r\ge2$, and let $0<p_i<1$ for $i\in[r]$. Let $q_0:=\max_{i\in[r]}p_i$, $L_i:=-\log p_i$ and     $L:=\sum_{i=1}^r L_i$. Suppose \(s:=\left\lceil\frac{1}{1-q_0}\right\rceil\le r\) and $
    \max_{i\in[r]}L_i\le \frac{L}{s}.$
Then every \(r\)-cross-intersecting tuple
$    \F_1,\ldots,\F_r\subseteq2^{[n]}$
satisfies
\[
    \prod_{i=1}^r\mu_{p_i}(\F_i)
    \le
    \prod_{i=1}^r p_i.
\]
\end{cor}

\begin{proof}
If some \(\F_i\) is empty, the result is immediate.  We may therefore assume that all
\(\F_i\)'s are non-empty.  As $s\le r$, every \(s\)-subcollection is \(s\)-cross-intersecting. For each \(i\in[r]\), let \(\G_i:=\F_i^\uparrow\).  Since
$q_0\le \frac{s-1}{s},$
Theorem~\ref{thm:main}, applied to every \(s\)-subcollection, gives
    $\prod_{i\in E}\mu_{q_0}(\G_i)\le q_0^s$
    for every $E\in\binom{[r]}s.$
If \(\mu_{q_0}(\G_i)=0\) for some \(i\), then \(\mu_{p_i}(\F_i)=0\), and there is nothing to
prove.  Thus assume \(\mu_{q_0}(\G_i)>0\) for all \(i\).  Define $a_i:=\frac{\log\mu_{q_0}(\G_i)}{\log q_0}.$
Since \(0<q_0<1\), the preceding inequalities are equivalent to
$\sum_{i\in E}a_i\ge s$
for every $E\in\binom{[r]}s.$

Next write
    $p_i=q_0^{\alpha_i}$ and $
    \alpha_i:=\frac{\log p_i}{\log q_0}.$
By Lemma~\ref{lem:comparison},
$\mu_{p_i}(\F_i)
    \le
    \mu_{p_i}(\G_i)
    \le
    \mu_{q_0}(\G_i)^{\alpha_i}
    =
    q_0^{\alpha_i a_i}.$
Thus it is enough to prove
$    \sum_{i=1}^r\alpha_i a_i\ge \sum_{i=1}^r\alpha_i.$
Let \(A:=\sum_i\alpha_i\).  The logarithmic balance assumption is exactly
    $\max_i\alpha_i\le \frac{A}{s}.$
Set
$    w_i:=\frac{s\alpha_i}{A}.$
Then \(0\le w_i\le1\) and \(\sum_iw_i=s\).  Hence
\((w_1,\ldots,w_r)\) lies in the hypersimplex
    $\left\{w\in[0,1]^r:\sum_iw_i=s\right\},$
which is the convex hull of the incidence vectors of the \(s\)-subsets of \([r]\).  Therefore
there are non-negative coefficients \(\lambda_E\), \(E\in\binom{[r]}s\), with
\(\sum_E\lambda_E=1\), such that
    $w=\sum_{E\in\binom{[r]}s}\lambda_E\mathbbm{1}_E.$
It follows that
\[
    \sum_{i=1}^r\alpha_i a_i
    =
    \frac{A}{s}\sum_{i=1}^r w_i a_i  =
    \frac{A}{s}
    \sum_{E\in\binom{[r]}s}\lambda_E\sum_{i\in E}a_i  \ge
    \frac{A}{s}
    \sum_{E\in\binom{[r]}s}\lambda_E s
    =
    A.
\]
Consequently,
$\prod_{i=1}^r\mu_{p_i}(\F_i)
    \le
    q_0^{\sum_i\alpha_i a_i}
    \le
    q_0^{\sum_i\alpha_i}
    =
    \prod_{i=1}^r p_i.$
\end{proof}

The condition in Corollary~\ref{cor:log-balanced-unequal} is genuinely unequal-bias when
\(q_0\) is bounded away from the critical value \((r-1)/r\).  For instance, if \(q_0\le1/2\),
then \(s=2\), and the condition becomes
    $\max_i(-\log p_i)
    \le
    \frac12\sum_{j=1}^r(-\log p_j).$
However, in the top range
    $\frac{r-2}{r-1}<q_0\le\frac{r-1}{r},$
we have \(s=r\), and the balance condition forces
    $p_1=\cdots=p_r.$
Thus this corollary does not reach the hardest part of the Frankl--Tokushige conjecture; a
new unequal-bias ingredient seems necessary there.

\subsection{Index-hypergraph product bound}

We can use Theorem~\ref{thm:main} to prove a general index-hypergraph formulation of the product bound.  
The hypergraph records which subcollections are required to be cross-intersecting, and the exponent in the upper bound is determined by the corresponding fractional cover linear program.
\begin{theorem}[Index-hypergraph product bound]\label{thm:index-hypergraph-product}
Let $\mathcal H$ be a non-empty hypergraph on vertex set $[r]$, all of whose edges have size at least $2$.  Let $\F_1,\dots,\F_r\subseteq2^{[n]}$, and suppose that
$0\le p\le \min_{E\in\mathcal H}\frac{|E|-1}{|E|}.$
Suppose that for every edge $E\in\mathcal H$, the subcollection $\{\F_i:i\in E\}$ is $|E|$-cross-intersecting.  Define
\[
\tau(\mathcal H):=\min\left\{\sum_{i=1}^r a_i:
        a_i\ge0\text{ for all }i,
        \ \sum_{i\in E}a_i\ge |E|\text{ for every }E\in\mathcal H
        \right\}.
\]
Then
\[
\prod_{i=1}^r \mu_p(\F_i)\le p^{\tau(\mathcal H)}.
\]
\end{theorem}
\begin{proof}
If $p=0$, then, for any edge $E\in\mathcal H$, the families $\{\F_i:i\in E\}$ are $|E|$-cross-intersecting, so they cannot all contain $\varnothing$. Hence at least one family has $\mu_0$-measure $0$, and the product is $0$. Since $\mathcal H$ is non-empty, we have $\tau(\mathcal H)>0$, and therefore $p^{\tau(\mathcal H)}=0$.

Now assume $0<p<1$. If $\mu_p(\F_i)=0$ for some $i$, the conclusion is immediate. Otherwise define   $a_i:=\frac{\log\mu_p(\F_i)}{\log p},$ so that $
    \mu_p(\F_i)=p^{a_i}.$
Since $0<\mu_p(\F_i)\le1$ and $0<p<1$, we have $a_i\ge0$. For every edge $E\in\mathcal H$, Theorem~\ref{thm:main}, applied to the $|E|$-cross-intersecting subcollection $\{\F_i:i\in E\}$, gives
    $\prod_{i\in E}\mu_p(\F_i)\le p^{|E|}.$
Equivalently, $\sum_{i\in E}a_i\ge |E|$. Thus $(a_1,\dots,a_r)$ is feasible in the linear program defining $\tau(\mathcal H)$, and hence $\sum_i a_i\ge\tau(\mathcal H)$. Therefore
$    \prod_{i=1}^r\mu_p(\F_i)=p^{\sum_i a_i}\le p^{\tau(\mathcal H)},$
as required.
\end{proof}

There is a fractional matching interpretation.
By linear programming duality,
\[
\tau(\mathcal H)=\max\left\{\sum_{E\in\mathcal H}|E|y_E:
        y_E\ge0\text{ for all }E,
        \ \sum_{E\ni i}y_E\le1\text{ for every }i\in[r]
        \right\}.
\]
In particular, if $\mathcal H$ is $q$-uniform, then $\tau(\mathcal H)=q\nu^*(\mathcal H)$, where $\nu^*(\mathcal H)$ is the fractional matching number of $\mathcal H$.

\begin{cor}\label{cor:index-hypergraph-alpha}
Let $\mathcal H$ be a non-empty $q$-uniform hypergraph on $[r]$ with independence number $\alpha(\mathcal H)$, where $q\ge2$. Suppose that $0\le p\le (q-1)/q$ and that, for every $E\in\mathcal H$, the families $\{\F_i:i\in E\}$ are $q$-cross-intersecting. Then
\[
    \prod_{i=1}^r\mu_p(\F_i)\le p^{r-\alpha(\mathcal H)+q-1}.
\]
\end{cor}
\begin{proof}
It suffices to prove $\tau(\mathcal H)\ge r-\alpha(\mathcal H)+q-1$. Let $(a_1,\dots,a_r)$ be any feasible vector in the definition of $\tau(\mathcal H)$, and reorder the coordinates so that $a_1\le\cdots\le a_r$.  Put $\alpha:=\alpha(\mathcal H)$.  Since every set of $\alpha+1$ vertices contains an edge of $\mathcal H$, the set $\{1,\dots,\alpha+1\}$ contains some edge $E$.  As $\mathcal H$ is $q$-uniform, $|E|=q$, and hence
$    \sum_{j=\alpha-q+2}^{\alpha+1}a_j\ge \sum_{i\in E}a_i\ge q.$
In particular, $a_{\alpha+1}\ge1$, and therefore
\[
    \sum_{i=1}^r a_i
    \ge \sum_{j=\alpha-q+2}^{\alpha+1}a_j+\sum_{j=\alpha+2}^{r}a_j
    \ge q+(r-\alpha-1)=r-\alpha+q-1.
\]
Taking the minimum over all feasible vectors gives $\tau(\mathcal H)\ge r-\alpha(\mathcal H)+q-1$.  The result now follows from Theorem~\ref{thm:index-hypergraph-product}.
\end{proof}

A simple special case of Corollary~\ref{cor:index-hypergraph-alpha} resembles the classical
\((p,q)\)-property in discrete geometry: among any $p$ sets in a family, some $q$ of them have a common point. The Hadwiger--Debrunner problem and the Alon--Kleitman $(p,q)$-theorem study the global piercing consequences of this local intersection assumption~\cite{AK1992,HD1957}. Suppose that among every \(k\) of the families there are \(q\) which are
\(q\)-cross-intersecting.  Define the \(q\)-uniform index hypergraph
$\mathcal H:=
    \left\{
        Q\in\binom{[r]}q:
        \{\F_i:i\in Q\}\text{ is }q\text{-cross-intersecting}
    \right\}.$
Then every \(k\)-subset of \([r]\) contains an edge of \(\mathcal H\), so
    $\alpha(\mathcal H)\le k-1.$
Hence Corollary~\ref{cor:index-hypergraph-alpha} gives, for
\(p\le (q-1)/q\),
\[
    \prod_{i=1}^r\mu_p(\F_i)
    \le
    p^{r-\alpha(\mathcal H)+q-1}
    \le
    p^{r-k+q}.
\]
This exponent is tight: take \(k-q\) families to be \(2^{[n]}\) and the remaining
\(r-k+q\) families to be the same \(1\)-star.

\subsection{Further questions}

We end with several directions suggested by the proof.

\begin{itemize}
    \item \emph{Unequal biases.}
    The unequal-bias case of the Frankl--Tokushige conjecture remains open.
    It would be interesting to find a genuinely multi-bias analogue of the critical-bias
    coupling or of the additive estimate in Lemma~\ref{lem:critical-sum}.

    \item \emph{One-sided stability.}
    Theorem~\ref{thm:r-cross-product-stability} gives a symmetric-difference bound
    $\mu_p(\F_i\Delta \mathcal S_j)\le C_r\varepsilon p.$
    The deletion example shows that a linear bound is necessary for symmetric difference.
    However, it may still be possible to prove a stronger one-sided estimate for the mass
    outside the approximating star, $\mu_p(\F_i\setminus\mathcal S_j),$
    separating ``added'' sets outside the star from ``deleted'' sets inside the star.

    \item \emph{\(t\)-wise common intersection.}
    What is the correct product theorem for families satisfying
     $   |A_1\cap\cdots\cap A_r|\ge t$
    for every \(A_i\in\F_i\)?  The present coordinatewise coupling forces the common
    intersection to be empty and is therefore naturally adapted to the case \(t=1\).  The
    case \(t>1\) likely requires a different coupling or a different critical inequality.
\end{itemize}

\noindent\textbf{Acknowledgements.} During an early exploratory stage, the authors used language-model-based tools to help brainstorm candidate proof strategies; any suggestions arising from these tools served only as informal inspiration.

\bibliographystyle{abbrv}
\bibliography{reference}

\end{document}